\documentclass[11pt]{amsart}

\usepackage{amsmath,amsfonts,amsthm,amsopn,amssymb}
\usepackage{verbatim,wasysym,mathrsfs}
\usepackage[left=2.6cm,right=2.6cm,top=3.1cm,bottom=3.1cm]{geometry}
\usepackage{cite,marginnote}
\usepackage{color,enumitem,graphicx}
\usepackage[colorlinks=true,urlcolor=blue, citecolor=red,linkcolor=blue,
linktocpage,pdfpagelabels, bookmarksnumbered,bookmarksopen]{hyperref}
\usepackage[english]{babel}
\usepackage[T1]{fontenc}
\usepackage{lmodern}
\usepackage[hyperpageref]{backref}

\numberwithin{equation}{section}

\makeindex
\newtheorem{theorem}{Theorem}[section]
\newtheorem{proposition}[theorem]{Proposition}
\newtheorem{lemma}[theorem]{Lemma}
\newtheorem{remark}[theorem]{Remark}
\newtheorem{example}[theorem]{Example}
\newtheorem{corollary}[theorem]{Corollary}
\newtheorem{definition}[theorem]{Definition}

\newcommand{\R}{\mathbb R}

\newcommand{\bt}{\begin{theorem}}
\newcommand{\et}{\end{theorem}}
\newcommand{\bl}{\begin{lemma}}
\newcommand{\el}{\end{lemma}}
\newcommand{\bd}{\begin{definition}}
\newcommand{\ed}{\end{definition}}
\newcommand{\bc}{\begin{corollary}}
\newcommand{\ec}{\end{corollary}}
\newcommand{\bp}{\begin{proof}}
\newcommand{\ep}{\end{proof}}
\newcommand{\bx}{\begin{example}}
\newcommand{\ex}{\end{example}}
\newcommand{\bi}{\begin{exercise}}
\newcommand{\ei}{\end{exercise}}
\newcommand{\bo}{\begin{proposition}}
\newcommand{\eo}{\end{proposition}}
\newcommand{\br}{\begin{remark}}
\newcommand{\er}{\end{remark}}
\newcommand{\be}{\begin{equation}}
\newcommand{\ee}{\end{equation}}
\newcommand{\ba}{\begin{align}}
\newcommand{\ea}{\end{align}}
\newcommand{\bn}{\begin{enumerate}}
\newcommand{\en}{\end{enumerate}}
\newcommand{\bg}{\begin{align*}}
\newcommand{\bcs}{\begin{cases}}
\newcommand{\ecs}{\end{cases}}

\def\R{\mathbb R}

\def\qed{\hfill$\square$\smallskip}
\makeatletter
\def\@makefnmark{}
\makeatother

\newcommand{\bean}{\begin{eqnarray*}}
\newcommand{\eean}{\end{eqnarray*}}

\renewcommand{\epsilon}{\varepsilon}

\newcommand{\eqnn}[1]{\begin{equation}\begin{split}#1\end{split}\nonumber\end{equation}}
\newcommand{\beq}{\begin{equation}}
\newcommand{\eeq}{\end{equation}}
\newcommand{\bea}{\begin{eqnarray}}
\newcommand{\eea}{\end{eqnarray}}

\newcommand{\eq}[2]{\begin{equation}\begin{split}#1\end{split}\label{#2}\end{equation}}

\graphicspath{{figures/}}

\newcommand{\blue}[1]{{\color{blue}{#1}}}

\newcommand\LT{L^2(\mathbb{R})}

\newcommand\HT{H^1(\mathbb{R})}

\title[Orbital stability of smooth solitary waves]{Orbital stability of smooth solitary waves for the modified Camassa-Holm equation}

\author[X. J. Deng]{Xijun Deng}
\author[S. Lafortune]{St\'{e}phane Lafortune}
\author[Z. S. Liu]{Zhisu Liu}

\address[X. J. Deng]{\newline\indent School of Mathematics,
\newline\indent
Hubei University of Automotive Technology,
Shiyan, Hubei, 442002, P. R. China}
\email{\href{mailto:xijundeng@yeah.net}{xijundeng@yeah.net}}

\address[S. Lafortune]{\newline\indent Department of Mathematics,
\newline\indent
College of Charleston,
 Charleston, SC 29401, USA}
\email{\href{mailto:lafortunes@cofc.edu}{lafortunes@cofc.edu}}

\address[Z. S. Liu]{\newline\indent  {{School of Mathematics and Physics,}}
\newline\indent
China University of Geosciences,
Wuhan, Hubei, 430074, P. R. China}
\email{\href{mailto:liuzhisu@cug.edu.cn}{liuzhisu@cug.edu.cn}}

\subjclass[2010]{}
\thanks{Corresponding author: S. Lafortune ({\tt lafortunes@cofc.edu})}

\begin{document}

\begin{abstract}
In this paper, we explore the orbital stability of smooth solitary wave solutions to
the modified Camassa-Holm equation with cubic nonlinearity. These solutions, which exist
 on a nonzero constant background $k$, are unique up to translation for each permissible value
 of $k$ and wave speed. By leveraging the Hamiltonian nature of the modified Camassa-Holm equation
 and employing three conserved functionals-comprising an energy and two Casimirs, we establish orbital
 stability through an analysis of the Vakhitov-Kolokolov condition.
This stability pertains to perturbations of the momentum variable in $H^1(\mathbb{R})$.
\end{abstract}
\maketitle

	\maketitle

\section{Introduction}
\label{sec1}

\label{1s}
In this paper, we are concerned with the stability of smooth solitary waves for the following modified Camassa-Holm (mCH)
equation with cubic nonlinearity
\begin{equation} \label{1} m_t+((u^2-u_x^2)m)_x=0, \;\; m=u-u_{xx},
\end{equation} where $u(t,x)$
is a real-valued function of space-time variables $(x, t)$, and the subscripts $x$ and $t$ appended to $m$ and $u$ denote partial
differentiation. The variable $m$ is deemed the ``momentum variable'' in the realm of peakon equations. The mCH equation \eqref{1}
was proposed as an integrable modified version
of the Camassa-Holm equation (CH)
\begin{equation} \label{2} m_t+(um)_x+u_xm=0, \;\; m=u-u_{xx}.
\end{equation}
It was derived using a general approach of the tri-Hamiltonian duality by Fuchssteiner \cite{FU96} and Olver
and Rosenau \cite{OR96}. Also, Equation \eqref{1} is integrable in the sense that it has a bi-Hamiltonian
structure \cite{FU96, OR96, Qiao06} and it admits a Lax pair \cite{Qiao06}. Later, the generalization of \eqref{1}
with a dispersive term given below in \eqref{3} was obtained by Qiao \cite{Qiao11} from the two-dimensional Euler
equations with the variable $u$ representing the velocity of the fluid. Qiao also obtained the bi-Hamiltonian structure
together with the Lax pair. Furthermore,
it is shown in \cite{GL13} that Equation \eqref{3} arises from an intrinsic invariant planar curve flow in Euclidean geometry.

Equation \eqref{1} admits a variety of solutions on the whole line, for both the cases of zero
 and nonzero backgrounds. More specifically, \eqref{1} admits peaked solitary waves (peakons) \cite{GL13},
 which are asymptotically going to zero, whereas the smooth multisoliton solutions \cite{IL12,M13} tend to a nonzero
 constant as $x \to \pm \infty$. Orbital stability of the single peakons for the mCH equation \eqref{1} was obtained
 in \cite{QLL13} using an approach similar to what was done in \cite{CS00, LL09} for the Camassa-Holm and Degasperis-Procesi
 one-peakon solutions. As for the train of peakons of the mCH equation \eqref{1}, their orbital stability was obtained in
 \cite{LLQ14} by using an energy argument and combining the method of the orbital
stability of a single peakon with the monotonicity of the local energy norm \cite{MMT02, KM20}.

In \cite{LLZ24}, Li, Liu, and Zhu investigated the stability of smooth solitary-wave solutions for the mCH equation with a linear dispersion term, given by
\begin{equation}\label{3}
m_t + ((u^2 - u_x^2)m)_x + \gamma u_x = 0,\;\;\; m = u - u_{xx},
\end{equation} where $\gamma$ is a positive constant. Unlike the mCH equation \eqref{1}, which does not admit smooth traveling wave solutions with vanishing boundary conditions, Equation \eqref{3} admits smooth soliton solutions that vanish at infinity \cite{M14, LL21}. By constructing conserved quantities in terms of the momentum variable $m$, it is shown in \cite{LLZ24} that the smooth soliton for Equation \eqref{3} is orbitally stable to perturbations to $m$ in $\HT$. As discussed in detail in \cite[Section 1]{LLZ24}, it would be challenging to apply the framework developed by Grillakis, Shatah, and Strauss in \cite{GSS87} or the methods used for the Camassa-Holm equation in \cite{CS02}. This is due to the cubic nonlinearity of the mCH equation, which causes the integrands of the conserved functionals to be quartic. Similar difficulties arise in establishing the orbital stability of smooth solitons of the Degasperis-Procesi equation (see the discussions in the introductions of \cite{LLW20, LLW22, LLW23}). Therefore, an approach following entirely the method outlined in \cite{GSS87} based on the Hamiltonian structure in the variable $u$ with invertible Hamiltonian operator does not seem to be the best of choices for the mCH equation \eqref{1}.

Inspired by the works in \cite{HL13, LP22, EJL24}, we aim to study the stability of smooth solitary wave solutions of the mCH equation \eqref{1} using a Hamiltonian structure, written in the variable $m$, that admits {{Casimirs$^1$}} \footnote{$^1$Casimirs are conserved functionals that have a zero Poisson bracket with all other functionals, including the Hamiltonian itself}. Our general approach is as follows. First, we establish the existence of smooth solitary waves for \eqref{1} under nonzero boundary conditions. We then demonstrate that such waves can be considered as critical points of an appropriate action functional expressed in the variable $m$, constructed out of a linear combination of the Hamiltonian and the Casimirs, similarly as done for the $b$-family in \cite{HL13, LP22, LL24} and for the Novikov equation in \cite{EJL24} (with the Casimirs for the Novikov equation given in \cite[Section 3]{Hone08}). Note that the $b$-family is a peakon equation that has a free parameter denoted by $b$ \cite{dhh, Dullin}, which includes two integrable cases: the Camassa--Holm equation \cite{CH93} ($b=2$) and the Degasperis--Procesi equation \cite{dp} ($b=3$). Novikov proposed the integrable Novikov peakon equation in \cite{Novikov}, and it is regarded as a generalization of the Camassa-Holm equation that accounts for cubic nonlinearities.

Next, we aim to determine conditions that ensure such solitary waves are constrained local extrema of the Hessian of the associated action functional. This strategy is in line with the Energy-Casimir method \cite{ecmethod}, except that in the next step, one finds that the Hessian operator has negative spectrum.
Nevertheless, after establishing the spectral properties of the Hessian operator, we show that these waves are orbitally stable in the constrained space, provided that a so-called Vakhitov-Kolokolov condition \cite{VK73} is satisfied. Finally, we verify that the Vakhitov-Kolokolov condition always holds, indicating that the smooth solitary wave solutions of the mCH equation \eqref{1} are orbitally stable, a consequence of the general stability result established in \cite{GSS87}.

Notice that there exist no smooth solitary waves for the mCH equation \eqref{1} with vanishing boundary condition. Therefore, we consider the solutions of \eqref{1} on a nonzero constant background with $m(t,x)\rightarrow k$ as $x\rightarrow \pm \infty$. Moreover, for fixed $k>0$, we consider the class of functions in the set
\begin{equation}\label{4} X_k:=\{m-k\in H^1(\mathbb{R}):m(x)>0,\,x\in \mathbb{R}\}. \end{equation}
The three conserved integrals we are using to construct the Lyapunov functional are given by
\begin{equation}\label{5} F_1(m)=\int_{\mathbb{R}}(m-k)\,dx \end{equation}
\begin{equation}\label{6} F_2(m)=\int_{\mathbb{R}}\left(\frac{1}{m}-\frac{1}{k}\right)\,dx \end{equation}
and
\begin{equation}\label{7}
F_3(m)=\int_{\mathbb{R}}\left(\frac{m_x^2}{m^5}+\frac{1}{4m^3}-\frac{1}{4k^3}\right)\,dx.
 \end{equation}
The fact that $F_1$ is constant is evident. The conserved integrals $F_2$ and $F_3$ were derived in \cite{M13} using a B\"{a}cklund transformation. However, since it is not immediately apparent that $F_2$ and $F_3$ are conserved, we verify this fact explicitly  in {{Appendix \ref{AA}}}.

%By using the three conserved integrals $F_1$, $F_2$, and $F_3$, we can establish the main stability result presented in the following theorem.

The equation \eqref{1} can be written in {{Hamiltonian}} form with energy given by $F_1$ above as \cite[Section 2]{Olver2016}
\begin{equation}
m_t={\mathcal{J}}\frac{\delta F_1}{\delta m},\;\;{\mathcal{J}} \equiv \partial_x m \partial_x^{-1} m (\partial_x^2-1)^{-1}  \partial_x m \partial_x^{-1} m \partial_x.
\label{HM}
\end{equation}
In {{Appendix \ref{AA}}}, we show that $F_2$ and $F_3$ are Casimirs for the Hamiltonian system described above, thus providing an alternative proof that $F_2$ and $F_3$ are conserved quantities. Also, by doing so,
we show that the Lyapunov functional is obtained by taking the energy functional and a linear combination of the Casimirs, as done in \cite{HL13,LP22,EJL24} for smooth solutions to the $b$-family and the Novikov equations.

The main result of this article is given in the following theorem.

\begin{theorem}\label{th1.1}  For fixed $c>0$, and $k\in(\frac{\sqrt{c}}{3}, \frac{\sqrt{3c}}{3})$, there
exists a unique smooth solitary wave $m(t, x) =\mu (x - ct)$ of the mCH equation \eqref{1}. This {{solitary}} wave $\mu (x - ct)$ is orbitally stable in the space $X_k$ defined in \eqref{4}, namely, if for every $\varepsilon>0$ there exists $\delta=\delta(\varepsilon)>0$ such
that for every $m_0\in X_k$ satisfying $\|m_0-\mu(\cdot)\|_{{H^1}}<\delta$, there exists
a unique solution $m\in C^0(\mathbb{R}, X_k)$ of the mCH equation \eqref{1} with the initial datum
$m(0, \cdot) = m_0$ and the maximal existence time $T>0$ satisfying
$$
{{\inf_{r\in \mathbb{R}}\left\{\|m(t,\cdot)-\mu(\cdot+r)\|_{{H^1}}, t\in [0,T)\right\}<\varepsilon.}}
$$\\
\end{theorem}

\begin{remark} \label{re1.0}
We can apply the global existence result from Proposition \ref{pro.2} below by imposing the additional conditions that {\( m_0 - k \in H^{2,1}(\mathbb{R}) \cap H^{1,2}(\mathbb{R})
\subset H^2(\mathbb{R}) \) }and that \( m_0 > 0 \), where {\( H^{p,s}(\mathbb{R}) \)} is defined in \eqref{Hps}.
Under these assumptions, we may replace \( T \) with \( \infty \) in Theorem \ref{th1.1}.
\end{remark}

\begin{remark} \label{re1.1}
Note that
$$
 \|m\|^2_{H^1}=\int_{\mathbb{R}}(m^2+m_x^2)dx=\int_{\mathbb{R}}(u^2+3u_x^2+3u_{xx}^2+u_{xxx}^2)\,dx
$$
which is equivalent to the $H^3(\mathbb{R})$ norm on $u$. {{Thus it follows from Theorem \ref{th1.1} that if $u(t, x) = \phi(x-ct)$ is the
traveling wave solution of Equation \eqref{1}, we can define the set given by
\eqnn{
Y_k:=\{u-k\in H^3(\mathbb{R}):u-u_{xx}>0, x\in \mathbb{R}\}.
}
Then, by Theorem \ref{th1.1}, such wave is orbitally stable in $Y_k$ in the $H^3(\mathbb{R})$ norm in the following sense. If $u_0\in Y_k$ and $\|u_0-\phi(\cdot)\|_{{H^3}}<\delta$ then
$$
\inf_{r\in \mathbb{R}}\left\{\|u(t,\cdot)-\phi(\cdot+r)\|_{H^3}, t\in [0,T)\right\}<\varepsilon.
$$}}
\end{remark}

\begin{remark} \label{re1.2}
 The existence and uniqueness of the smooth solitary wave in
Theorem \ref{th1.1} will be established in Lemma \ref{3.1}.
\end{remark}

\begin{remark} \label{re1.3}
 The solutions studied in \cite{LLZ24} for Equation \eqref{3} are on a zero background, and thus are not related to the solitary wave with non-zero background for the mCH equation \eqref{3} with $\gamma=0$ (i.e. equation \eqref{1}). Hence, our results  are not obtainable by simply taking the limit as $\gamma\to 0$.
 %Moreover, it also shows that Theorem \ref{th1.1} is completely different from the corresponding main result presented in \cite{LLZ24}.
 \end{remark}

\begin{remark} \label{re1.4}
Equation \eqref{3} possesses, among others, the following two conserved integrals \cite{M13}, which were used in \cite{LLZ24} to prove the orbital stability of the smooth solitary wave solutions when $\gamma \neq 0$: \eqnn{ E(m)=\int_{\mathbb{R}}\left(M-\sqrt{\frac{\gamma}{2}}\right)dx, \ \ F(m)=\int_{\mathbb{R}}\left(\frac{m_x^2}{2M^5}-\frac{1}{\gamma M}+\sqrt{2}\gamma^{-\frac{3}{2}} \right)dx, } where $M=\sqrt{m^2+\frac{\gamma}{2}}$. While the expression for $F$ above is similar to the function $F_3$ from \eqref{7} used in the case $\gamma=0$, the limit $\gamma \to 0$ cannot be performed on $F$. This suggests that $F_3$ is a conserved functional that does not arise as a reduction of a functional from the case $\gamma > 0$.
  \end{remark}

{{While the methodology used in this article for analyzing smooth solitary waves of the modified Camassa-Holm (mCH) equation closely parallels the approach in \cite{LP22} for the $b$-family, there are important distinctions to highlight. In \cite{LP22}, the existence of solutions is established by leveraging the fact that the traveling wave equation can be interpreted as the conservation of kinetic and potential energy for a Newtonian particle. In contrast, this approach was not applicable in our case, {{so}} we instead employ phase plane analysis to demonstrate the existence of solitary wave solutions.
Furthermore, \cite{LP22} presents a stability criterion, but it is only verified analytically for the two integrable cases{{--$b=2$ (Camassa-Holm) and $b=3$ (Degasperis-Processi)--}}and asymptotically for two limiting parameter values (the criterion was later verified in \cite{LL23}). In our work, we are able to verify the stability criterion analytically for all parameter values.
Additionally, we have included explanations not present in \cite{LP22}, such as a discussion on the smoothness of solutions with respect to various parameters (see Section~\ref{S:VK}), and a justification of the invertibility of the second variation of the Lyapunov functional (also in Section~\ref{S:VK}).}}

The remainder of this paper is organized as follows. In Section \ref{2s}, we provide a brief review of the well-posedness of the mCH equation \eqref{1}.  In Section \ref{3s}, we examine the existence and fundamental properties of smooth solitary wave solutions of the mCH equation \eqref{1}. Specifically, we show that these solitary waves can be viewed as critical points of an explicit action functional $\Lambda$, and we then provide analytic spectral properties of the operator $\mathcal{L}=\frac{\delta^2 \Lambda}{\delta m^2}$ evaluated at these solitary waves. Section \ref{4s} is devoted to the analysis of orbital stability. By demonstrating that the Vakhitov-Kolokolov condition is always satisfied, we then complete the proof of Theorem \ref{th1.1}. {{Appendix \ref{AA} is devoted to the explicit proof that the integrals defined in \eqref{6} and \eqref{7} are conserved over time. We achieve this in two ways: first, by directly computing their time derivatives, and second, by showing that they are Casimirs for the Hamiltonian operator $\mathcal{J}$ defined in \eqref{HM}.}}

\section{{Well-posedness}}

\label{2s}
{{In this section, we recall the well-posedness results for the Cauchy problem of the mCH equation \eqref{1}}

% \subsection {Well-posedness}
% \label{2.1s}
 We consider the Cauchy problem of the mCH equation on the real line, that is,
\begin{equation}\left\{\begin{array}{l}\label{2.1}
m_t+\left((u^2-u_x^2)m\right)_x=0, \,\,\,\,m=u-u_{xx},\\ \\
u(0,x)=u_0(x), x\in \mathbb{R}.
\end{array}\right.\end{equation}\\

The following local well-posedness result and the properties of solutions on the line were established in \cite{GL13}

\begin{proposition} \label{pro.1}
Let $u_0\in H^s(\mathbb{R})$ with $s>\frac{5}{2}$. Then there exists a time $T>0$ such that the initial value problem \eqref{2.1} has a unique solution $u\in C([0,T),H^s(\mathbb{R}))\cap C^1([0,T),H^{s-1}(\mathbb{R}))$.
Moreover, the solution $u$ depends continuously on the initial data, and if $m_0=(1-\partial_x^2)u_0$ does not change sign, then $m(t,x)$ will not change sign for any $t\in[0,T)$. More precisely, if $m_0(x)>0$, then the corresponding solution $u(t,x)$ is positive and satisfies $|u_x(t,x)|\leq u(t,x)$ for $(t,x)\in [0,T)\times \mathbb{R}$.
\end{proposition}

In this article, we assume that the well-posedness result above can be extended
to the solutions with nonzero background of \eqref{1}. In other words, we
make the assumption that the tools used in \cite{GL13} to obtain Proposition \ref{pro.1} can
be extended to the equation obtained by making the substitution $u\rightarrow u + k$
into \eqref{1}.

Note that the following Proposition provides global existence of solutions on
a nonzero background \cite[Theorem 1.1]{YFL22}.

\begin{proposition} \label{pro.2}
Assume that the initial data $m(0) > 0$ and $m(0) -k \in H^{2,1}(\R)\cap  H^{1,2}(\R)$ for some $k>0$. Then there
exists a unique solution of the mCH \eqref{1} such that $m(t)-k \in C\left([0, +\infty);H^{2,1}(\R)\cap  H^{1,2}(\R)\right)$.
\end{proposition}
{\noindent}Here, by $H^{p,s}(\R)$, we mean \cite{YFL22}
\eq{
H^{p,s}(\R)=\left\{f\in\LT \,\vert\, (1+x^2)^{s/2}\partial_x^jf(x)\in \LT,\;j=0,1,...,p\right\}.
}{Hps}
As explained in Remark \ref{re1.0}, the proposition above can be used to reformulate
Theorem \ref{th1.1} into a \textquotedblleft  global stability\textquotedblright\ result.

\section {Smooth solitary wave solutions}
\label{3s}
In this section, we employ phase plane analysis to establish the existence of a one-parameter family of smooth solitary wave solutions for the mCH equation on a nonzero constant background. We then show that these solutions can be characterized variationally as the critical points of an action functional $\Lambda$. This functional is composed of a linear combination of the energy $F_1$ \eqref{5} and the Casimirs $F_2$ \eqref{6} and $F_3$ \eqref{7} {{(the functionals $F_2$ and $F_3$ are shown to be Casimirs in Appendix \ref{AA})}}. Finally, we derive the spectral properties of the Hessian operator $\mathcal{L}$ associated with the action functional $\Lambda$.

\subsection {Existence of smooth solitary wave solutions}
We consider the traveling wave solutions of \eqref{1} of  the form $u(t,x)=\phi(\xi)$, $\xi=x-ct$.  After integration, we find that  the profile $\phi$
satisfies the ODE
\begin{equation}\label{3.1}
(\phi-\phi_{\xi\xi})(\phi_\xi^2-\phi^2+c)=a,
\end{equation}
where $a$ is an integration constant.

The following lemma describes the family of solitary waves parameterized by the arbitrary nonzero background
parameter $k > 0$.

\begin{lemma} \label{le3.1}
 For fixed $c>0$, there exists a family of smooth solitary waves with $\phi\in C^{\infty}(\mathbb{R})$ satisfying $\phi(\xi)\rightarrow k$
as $|\xi|\rightarrow \infty$ if and only if $k\in(\frac{\sqrt{c}}{3}, \frac{\sqrt{3c}}{3})$. Moreover,
$$
k<\phi(\xi)\leq \sup_{\xi\in \mathbb{R}}{(\phi)}=\sqrt{2(c-k^2)}-k{{<\sqrt{c}}}, \ \ \mu=\phi-\phi_{\xi\xi}>0, \ \ \xi\in \mathbb{R}.
$$
{{The solution described above is unique and an even function of $\xi$
with the added condition $\phi'(0) = 0$.}}
\end{lemma}

\begin{proof}
Equation \eqref{3.1} can be rewritten as the following two-dimensional dynamical system
\begin{equation}\left\{\begin{array}{l}\label{3.2}
 \phi_\xi=\psi, \\ \\
 (c+\psi^2-\phi^2)\psi_\xi=\phi(c+\psi^2-\phi^2)-a.
\end{array}\right.\end{equation}\\
Another invariant of \eqref{3.1} can be obtained by multiplying the equation by $\phi_\xi$, performing an integration, and use the definition of $\psi$ given in \eqref{3.2}, which yields
\begin{equation}\label{3.3}
(\phi^2-\psi^2)^2-2c(\phi^2-\psi^2)+4a\phi=E,
\end{equation}
where $E$ is another {{constant of integration}}.

A smooth solitary wave profile satisfying $\phi(\xi)\rightarrow k$ as $|\xi|\rightarrow \infty$ correspond to a homoclinic
orbit to the equilibrium point $(\phi,\psi) = (k, 0)$ of system \eqref{3.2}. Taking the limit
as $|\xi|\rightarrow \infty$ in \eqref{3.1} and \eqref{3.3} yields the relations
\begin{equation}\label{3.4}
 a=k(c-k^2), \ \ E=k^2(2c-3k^2).
\end{equation}

After substituting \eqref{3.4} into \eqref{3.3}, it is found that a homoclinic orbit to $(k, 0)$ corresponds a bounded connected component of the level curve, denoted by
\eqnn{
\Gamma: (\phi^2-\psi^2)^2-2c(\phi^2-\psi^2)+4k(c-k^2)\phi=k^2(2c-3k^2),
}
which is equivalent to
\begin{equation}\label{3.6}
\Gamma: (\phi^2-\psi^2-c)^2=(c-k^2)(c+3k^2-4k\phi).
\end{equation}
In order to have a homoclinic orbit to $(k, 0)$, the curve $\Gamma$ must intersect the $\phi$ axis at least one point different from $(k,0)$.
Setting $\psi=0$ and assuming that $\phi\neq k$, we find that the equation defining $\Gamma$ is equivalent to
\begin{equation}\label{3.7}
 (\phi+k)^2=2(c-k^2)\Rightarrow c>k^2, \ \ \phi_{1,2}=\pm \sqrt{2(c-k^2)}-k.
\end{equation}

On the other hand, it follows from \eqref{3.6} that
\begin{equation}\label{3.8}
 \psi^2=\phi^2-c \pm\sqrt{(c-k^2)(c+3k^2-4k\phi)}.
\end{equation}
In view of the fact that $c>k^2$, in order for  \eqref{3.8} to admit $(k,0)$ as a solution, one must choose the ``$+$'' branch and consider the curve defined by
\begin{equation}\label{3.9}
 \psi^2=\phi^2-c +\sqrt{(c-k^2)(c+3k^2-4k\phi)}.
\end{equation}
The RHS of \eqref{3.9} has at most three zeros (not counting multiplicity): $\phi=k$ and, possibly, $\phi=\phi_i, i=1,2$. It inherits a zero at $\phi=\phi_i$ from the solutions of \eqref{3.7}, and its value is distinct from $\phi=k$, if and only if {$\phi_i^2 - c < 0$.}
As a consequence, Equation \eqref{3.9} does not admit $(\phi_{2},0)$ as a solution since a simple calculation shows that
$$\phi_2^2-c=c-k^2+2k\sqrt{2(c-k^2)}>0.
$$
The RHS of equation \eqref{3.9} thus has $\phi=k$ as a double zero and $\phi=\phi_1$ as a simple one if {$\phi_1^2<c$.}
Furthermore, for equation \eqref{3.9} to define a {{closed}} curve that includes the fixed point $(k,0)$ and $(\phi_1,0)$, one needs the RHS of the equation to be positive on the open interval between $k$ and $\phi_1$.  Since the RHS is positive for $\phi<0$, has a double zero at $\phi=k$, and a distinct simple zero at $\phi=\phi_1$ (if {$\phi_1^2<c$}), Equation \eqref{3.9} defines a closed curve in the $\phi-\psi$ plane if and only if $\phi=\phi_1$ is a zero of RHS and $\phi_1>k$. So the conditions for \eqref{3.9} to define a closed
curve with the points $(k,0)$ and $(\phi_1,0)$ in it are
\eq{
\phi_1^2-c<0\text{ and } \phi_1>k,
}{phi1cond}
which leads to the conditions $k^2\in \left(\frac{c}{9},\frac{c}{3}\right)$.

Thus, in that case, \eqref{3.9} defines the  homoclinic orbit $$\left\{(\phi, \psi):\psi^2=\phi^2-c +\sqrt{(c-k^2)(c+3k^2-4k\phi)}, \phi>k\right\}$$ and the equilibrium point $(k, 0)$.

Finally, it follows from \eqref{3.9} and the condition given in \eqref{phi1cond} on the maximum value $\phi_1$ of the wave that
\begin{equation}\label{3.10}
\phi^2-\phi_\xi^2-c=\phi^2-\psi^2-c<0.
\end{equation}
By \eqref{3.1}, we can derive that $\phi-\phi_{\xi\xi}>0$ since $a=k(c-k^2)>0$. The statement about the supremum of $\phi$ {{given}} in the lemma is a consequence of the formula for $\phi_1$ made in \eqref{3.7}.

{{Since the dynamical system \eqref{3.2} is autonomous, we can ensure the solution is unique by the choice $\psi(0)=0$. Since that same dynamical system enjoys the symmetry $(\phi,\psi,\xi)\to (\phi,-\psi,-\xi)$, the unique solution is an even function of $\xi$. In fact, any solution for which $\psi(0)=0$ will be even due to the presence of that symmetry.}}
\end{proof}

\begin{remark} \label{re3.1}
 In the limiting case $k\rightarrow \frac{\sqrt{3c}}{3}$, the level curve $\Gamma$ will intersect the $\phi$ axis at the saddle point $(k,0)$. This implies that there exists no smooth solitary wave.
 In the limiting case $k\rightarrow \frac{\sqrt{c}}{3}$, $\phi^2-\phi_\xi^2-c$ approaches zero at the location of the maximum value of the solitary wave since $\phi_1\to c^-$.  It then follows from \eqref{3.1} that $\phi_{\xi\xi}\to -\infty$ at that same location.
 Thus, the solitary wave is no longer smooth.
\end{remark}

\begin{remark} \label{re3.2}
Let $\phi:=k+(1-\partial_x^2)^{-1}(\mu-k)$, then the following relation between $\mu$
and $\phi$ is obtained from \eqref{3.1}, \eqref{3.4}, and \eqref{3.9}:
\begin{equation}\label{3.11}
 \mu=k\sqrt{\frac{c-k^2}{c+3k^2-4k\phi}}.
\end{equation}

\end{remark}

\begin{remark} \label{re3.3}
Note that $k<\phi(\xi)\leq \phi_1=\sqrt{2(c-k^2)}-k$ (see \eqref{3.7}). It follows from \eqref{3.11} that
$$
k<\mu(\xi)\leq M_{c,k},
$$
where $M_{c,k}=\frac{k\sqrt{c-k^2}}{2\sqrt{2}k-\sqrt{c-k^2}}$. Hence, $M_{c,k}\rightarrow \infty$ as $k\rightarrow \frac{\sqrt{c}}{3}$.

\end{remark}

\begin{figure}[htb!]
\centering
\includegraphics[width=0.8\textwidth]{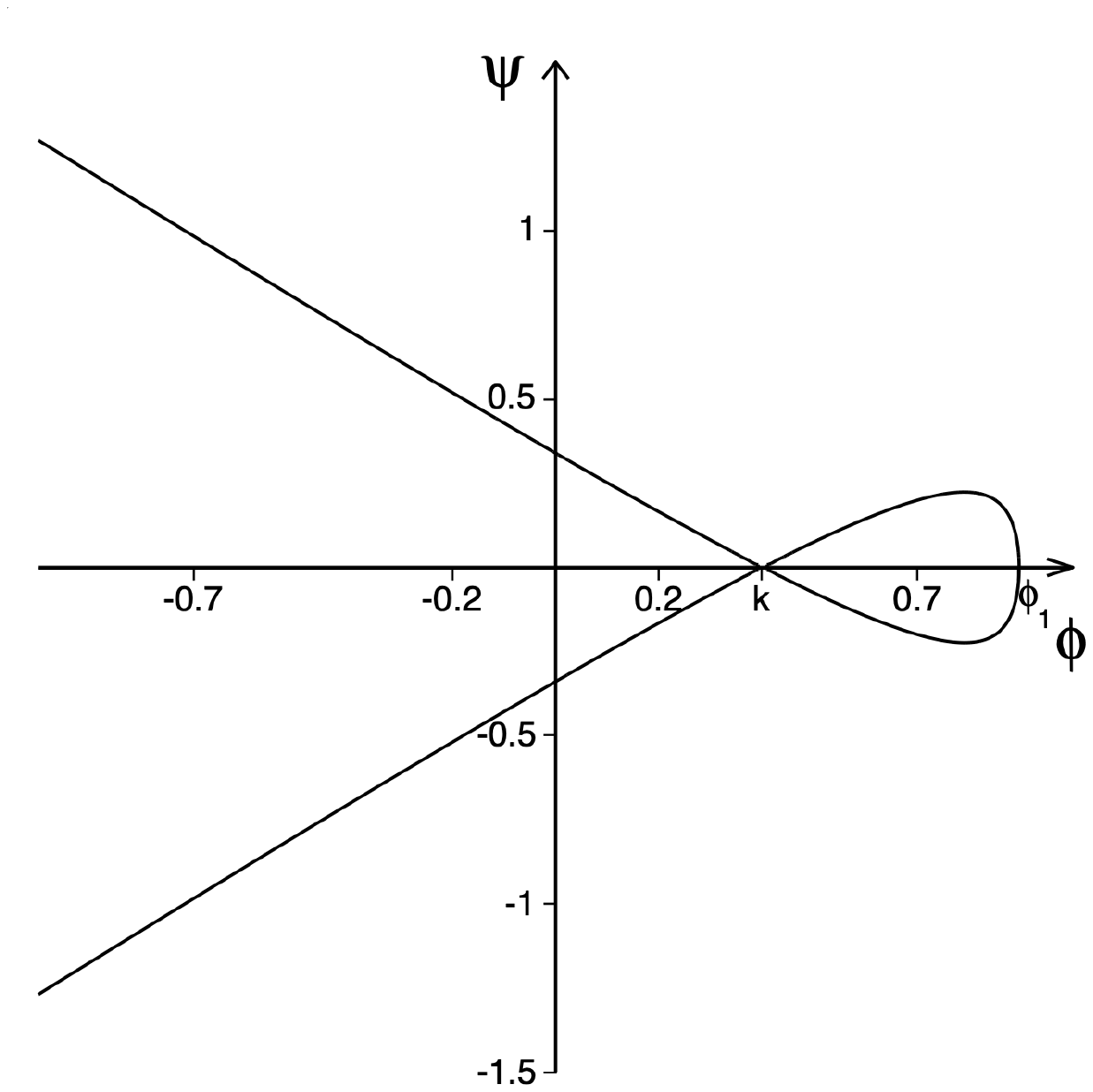}
\caption{Graph of the curve defined by Equation \eqref{3.9} in the case $k=0.4$ and $c=1$, corresponding to $\phi_1=0.896148140$.} \label{fig1}
\end{figure}

\subsection {Variational characterization}
We now show that the profile equations \eqref{3.1} and \eqref{3.3} satisfied by the smooth solitary waves
$\mu(.; k)$ correspond to the Euler-Lagrange equation of the action functional
\begin{equation}\label{3.12}
 \Lambda(m):=F_3(m)+\omega_1F_1(m)+\omega_2F_2(m), \ \ m\in X_k,
\end{equation}
where $m=u-u_{xx}$, and the conserved quantities $F_1, F_2$ and $F_3$ are given in \eqref{5}, \eqref{6} and \eqref{7}.
The following lemma states that the variational characterization is possible if
and only if the Lagrange multipliers $\omega_1$ and $\omega_2$ are uniquely related to parameters of Eqs.\eqref{3.1} and \eqref{3.3}.

\begin{lemma} \label{le3.2}
(Variational characterization) For any fixed $c > 0$ and $k\in(\frac{\sqrt{c}}{3}, \frac{\sqrt{3c}}{3})$, a critical point $\mu$ of
the action functional $\Lambda$ coincides with the solitary wave
solution $\mu= \phi-\phi_{\xi\xi}\in C^{\infty}(\mathbb{R})$, where $\phi$ is the unique positive solitary wave solution for the mCH equation \eqref{1},
if and only if
\begin{equation}\label{3.13}
\omega_1=\frac{c-9k^2}{4k^4(c-k^2)}, \quad \omega_2=-\frac{c+3k^2}{2k^2(c-k^2)}.
\end{equation}
\end{lemma}

\begin{proof}
Let $\mu\in X_k$ be a critical point of $\Lambda$. After straightforward
simplifications, the equation $\frac{\delta \Lambda}{\delta m}(\mu)= 0$ gives the
following differential equation
\begin{equation}\label{3.14}
\omega_1-\frac{\omega_2}{\mu^2}-\frac{2\mu_{\xi\xi}}{\mu^5}+\frac{5\mu_\xi^2}{\mu^6}-\frac{3}{4\mu^4}=0.
\end{equation}
 %Note that if $\mu$ is a weak solution to \eqref{3.13} then we have $\mu\in C^{\infty}(\mathbb{R})$,
 We want to show that the above differential equation \eqref{3.14} is satisfied whenever $\mu$ satisfies
the profile equation \eqref{3.1} and \eqref{3.3}, only if the Lagrange multipliers $\omega_1$ and $\omega_2$ are chosen as stated in \eqref{3.13}.

Using \eqref{3.1} and \eqref{3.3} we have
\begin{equation}\label{3.15}
\mu=\phi-\phi_{\xi\xi}=\frac{a}{\phi_\xi^2-\phi^2+c}, \ \ \phi^2-\phi_\xi^2-c=-\sqrt{c^2+E-4a\phi},
\end{equation}
since $\phi^2-\phi_\xi^2-c$ is negative (see \eqref{3.9}).

Thus, a direct calculation shows that
\begin{equation}\label{3.16}
 \mu_\xi=\frac{2a\phi_\xi\mu}{(\phi^2-\phi_\xi^2-c)^2}=2a^{-1}\phi_\xi\mu^3,
\end{equation}
and
\begin{equation}\label{3.17}
\mu_{\xi\xi}=2a^{-1}\phi_{\xi\xi}\mu^3+12a^{-2}\phi_\xi^2\mu^5.
\end{equation}
Substituting \eqref{3.15}, \eqref{3.16} and \eqref{3.17} into \eqref{3.14} yields
\eqnn{
 \omega_1-a^{-2}\omega_2(c^2+E-4a\phi)+2a^{-3}(c^2+E)\phi+4a^{-2}c-\frac{3}{4}a^{-4}(c^2+E)^2=0.
}
By grouping the constant terms and $\phi$ separately, this immediately gives the unique choice
for Lagrange multipliers
\begin{equation}\label{3.18}
 \omega_1=-4a^{-2}c+\frac{1}{4}a^{-4}(c^2+E)^2, \quad \omega_2=-\frac{1}{2}a^{-2}(c^2+E).
\end{equation}
Finally, substituting \eqref{3.4} into \eqref{3.18} gives \eqref{3.13}.
\end{proof}

Our next goal is to characterize the property of the critical point
$\mu$ as a local minimum, maximum or a saddle point. To this end, note that $\mu$ is uniformly bounded below by $k > 0$ by Remark \ref{re3.3},
we have $\mu+\widetilde{m}\in X_k$ if $\|\widetilde{m}\|_{H^1(\mathbb{R})}$ is sufficiently small.

The following result describes the second-order variation
of the action functional $\Lambda$. In what follows, $<.,.>$ denotes the
standard inner product in $L^2(\mathbb{R})$ space, and $\omega_1, \omega_2$ are selected by Lemma \ref{le3.2}.

\begin{corollary} \label{cor3.1}
 There exists a sufficiently small $\varepsilon>0$ such that for every $\widetilde{m}\in H^1(\mathbb{R})$ satisfying $\|\widetilde{m}\|_{H^1(\mathbb{R})}\leq \varepsilon$, we have
 \begin{equation}\label{3.19}
 \Lambda(\mu+\widetilde{m})-\Lambda(\mu)=\frac{1}{2}<\mathcal{L}\widetilde{m},\widetilde{m}>+R(\widetilde{m}),
\end{equation}
where $\mathcal{L}:=\frac{\delta^2 \Lambda}{\delta m^2}$ is given by
 \begin{equation}\label{3.20}
 \mathcal{L}=-\partial_\xi\mu^{-5}\partial_\xi+5\mu_{\xi\xi}\mu^{-6}-15\mu_\xi^2\mu^{-7}+\frac{3}{2}\mu^{-5}-\frac{c+3k^2}{2k^2(c-k^2)}\mu^{-3},
\end{equation}
and $\|R(\widetilde{m})\|_{H^1}\leq C_0\|\widetilde{m}\|^3_{H^1}$ for a positive constant $C_0$ independent of $\widetilde{m}$.
\end{corollary}

\begin{proof}
The expression for $\mathcal{L}$ can be obtained by using the straightforward Taylor expansion of $\Lambda(\mu+\widetilde{m})-\Lambda(\mu)$.
Since $\mu$ are strictly positive, bounded, and smooth on $\mathbb{R}$ by Lemma \ref{3.1} and Remark \ref{3.3}, we conclude that all the coefficients of the Sturm-Liouville operator $\mathcal{L}$ are smooth and bounded.

On the other hand, $R(\widetilde{m})$ can be computed by using the Taylor expansion of
the functionals $F_2(m)$ and $F_3(m)$ at $m =\mu$. Note that the leading order
term in $R(\widetilde{m})$ is cubic and the Sobolev space $H^1(\mathbb{R})$ forms a
Banach algebra with respect to multiplication, so $\|R(\widetilde{m})\|_{H^1}\leq C_0\|\widetilde{m}\|^3_{H^1}$
 follows from the Taylor expansion of $F_2(m)$ and $F_3(m)$ and the smallness of $\|\widetilde{m}\|_{H^1}$.
\end{proof}

\subsection {Spectral properties of $\mathcal{L}$}

The goal of this subsection is to show that the Hessian
operator $\mathcal{L}$ of the action functional $\Lambda$ defined by \eqref{3.19} and
\eqref{3.20} has exactly one simple negative eigenvalue and a simple zero
eigenvalue isolated from the rest of the spectrum.

Note that $c+\phi_\xi^2-\phi^2>0$, making the Liouville substitution
\eq{
z=\int_{0}^{\xi}\frac{1}{(c+\phi_x^2-\phi^2)^{\frac{5}{2}}}dx, \quad \omega(z)=(c+\phi_\xi^2-\phi^2)^{\frac{5}{4}}\upsilon(\xi),
}{TransfSL}
transforms the spectral equation $\mathcal{L}\upsilon=\lambda\upsilon$ into
 \begin{equation}\label{3.21}
 \mathcal{A}\omega(z)=\left(-a^{-5}\partial_z^2+\frac{3}{2}\mu^{-5}-\frac{c+3k^2}{2k^2(c-k^2)}\mu^{-3}+q(z)\right)\omega(z)=\lambda \omega(z),
\end{equation}
where
$$
q(z)=5\mu_{\xi\xi}\mu^{-6}-15\mu_\xi^2\mu^{-7}+\frac{5}{2}a^{-1}\mu^{-3}\phi_{\xi\xi}-\frac{35}{4}a^{-2}\mu^{-1}\phi_\xi.
$$
The transformation defined in \eqref{TransfSL} does not alter the spectrum. This is a consequence of the fact that the expression $c + \phi_\xi^2 - \phi^2$ limits to $c - k^2 > 0$ as $|\xi| \to \infty$ and, by \eqref{3.10}, it is strictly positive. Thus, the expression multiplying $\upsilon(x)$ in \eqref{TransfSL} is uniformly bounded away from zero, and the spectra of $\mathcal{A}$ and $\mathcal{L}$ are the same.

Both $\mathcal{L}$ and $\mathcal{A}$ are defined in $H^2(\mathbb{R})\subset H^1(\mathbb{R})$ dense in $L^2(\mathbb{R})$ and it can be verified that all the coefficients of both $\mathcal{L}$ and $\mathcal{A}$ are smooth and bounded. Thus they can be extended as unbounded self-adjoint Sturm-Liouville operators in $L^2(\mathbb{R})$. The following proposition gives the spectral properties of the linear
operator $\mathcal{L}$.

\begin{proposition} \label{pro3.1}
(1) The spectrum is real, $\sigma(\mathcal{L})\subset \mathbb{R}$;

(2) 0 is a simple isolated eigenvalue of $\mathcal{L}$ with $\mu_\xi$ as its eigenfunction;

(3) There exists $\delta> 0$ such that $\sigma(\mathcal{L})$ in $(-\infty, \delta)$ consists
of a simple zero eigenvalue and a simple negative eigenvalue;

(4) The essential spectrum is positive and given by $\sigma_{ess}(\mathcal{L})=\left[\frac{c-3k^2}{k^5(c-k^2)}, +\infty\right)$.

\end{proposition}

\begin{proof}
The linear operator $\mathcal{L}$ defined by \eqref{3.20} belongs to the
class of self-adjoint Sturm-Liouville operators in $L^2(\mathbb{R})$ with the
dense domain $H^2(\mathbb{R})$. Thus, $\sigma(\mathcal{L})\subset \mathbb{R}$. Moreover, the spectrum of the self-adjoint operator $\mathcal{L}$ on $\mathbb{R}\backslash\sigma_{ess}(\mathcal{L})$ contains only isolated simple eigenvalues.

By Sturm's Oscillation Theorem, it follows from the eigenvalue problem \eqref{3.21} that the point spectrum of $\mathcal{L}$ and $\mathcal{A}$ are the same.

Since $\mu(\xi)\rightarrow k$ as $|\xi|\rightarrow \infty$ exponentially fast, Weyl's Lemma states that the essential spectrum of $\mathcal{A}$ is given by the essential
spectrum of the linear operator with constant coefficients $\mathcal{A}_{\infty}$ given by
$$
\mathcal{A}_{\infty}=-a^{-5}\partial_{\xi\xi}+\frac{c-3k^2}{k^5(c-k^2)}.
$$
Since $c>3k^2$ (see Theorem \ref{th1.1}), the essential spectrum of $\mathcal{A}_{\infty}$ is $[\frac{c-3k^2}{k^5(c-k^2)},+\infty)=\sigma_{ess}(\mathcal{A})=\sigma_{ess}(\mathcal{L})$ which is strictly positive.

On the other hand, due to the translation symmetry of the mCH equation \eqref{1}, $\mu_\xi\in H^2(\mathbb{R})$ belongs
to the kernel of $\mathcal{L}$ so that 0 is in the spectrum of $\mathcal{L}$. Sturm's
Oscillation Theorem states that the $n$th simple eigenvalue corresponds
to the eigenfunction with $(n-1)$ simple zeros on $\mathbb{R}$. Since
$\mu_\xi=2a^{-1}\phi_\xi\mu^3$ has only one zero on $\mathbb{R}$, $0$ is the second eigenvalue of
$\mathcal{L}$ and there exists only one simple negative eigenvalue. (see for example \cite[Chapter 5]{TG12} for an exposition on
Sturm-Liouville Theory).
\end{proof}

\section{Orbital stability analysis}
\label{4s}
In this section, we first derive a stability criterion for the orbital stability of the solitary
wave $\mu$ with respect to perturbations in $H^1(\mathbb{R})$. Then we give an analytic proof that this stability criterion is satisfied.

Note that under the spectral properties of Proposition \ref{pro3.1}, Corollary \ref{cor3.1} implies that the solitary
wave $\mu$ is a degenerate saddle point of the action functional $\Lambda$.
The solitary waves can also be considered as
constrained critical points of $F_3$ subject to fixed $F_1$ and $F_2$ by the definition of $\Lambda$ in \eqref{3.12}. In this section, we first show that
$\mu$ is a constrained local minimizer of $\Lambda$, which implies stability by the result of \cite{GSS87}.

To understand precisely the appropriate constraint, notice that the action functional $\Lambda$ given in \eqref{3.12}, with Lagrange multipliers specified by {Lemma \ref{le3.2},} is written as a sum of three functionals. However, each of the functionals is not individually Fr\'echet differentiable in the sense that their derivatives taken with respect to the $\LT$ bracket are not in $\LT$ since
\eq{
\frac{\delta F_1}{\delta m}(\mu)=1, \quad \frac{\delta F_2}{\delta m}(\mu)=-\frac{1}{\mu^2}, \quad \frac{\delta F_3}{\delta m}(\mu)=-\frac{2\mu_{\xi\xi}}{\mu^5}+\frac{5\mu_\xi^2}{\mu^6}-\frac{3}{4\mu^4}.
}{FrD}
However, we can use the expressions above in \eqref{FrD} to rewrite $\Lambda(m)$  as the following sum of two Fr\'echet differentiable functionals
 \eq{
 {\Lambda(m)=\mathcal{G}(m)-\frac{c+3k^2}{2k^2(c-k^2)}{\mathcal{F}}(m),}
}{OmFr}
where
\eq{
\mathcal{G}(m):=F_3(m)+\frac{3}{4k^4}F_1(m),\;\;\mathcal{F}(m):=F_2(m)+k^{-2}F_1(m),
}{4.1}
with $F_i,\;i=1,2,3$ given in {\eqref{5}, \eqref{6}, and \eqref{7}.} {{We compute the Fr\'echet derivatives
\eq{
\frac{\delta \mathcal{G}}{\delta m}(\mu)=\frac{3}{4k^4}-\frac{2\mu_{\xi\xi}}{\mu^5}+\frac{5\mu_\xi^2}{\mu^6}-\frac{3}{4\mu^4},\quad \frac{\delta \mathcal{F}}{\delta m}(\mu)=\frac{1}{k^2}-\frac{1}{\mu^2}.
}{Frver}
Because the three limits, $\mu\to k$ and $\mu_\xi, \mu_{\xi\xi}\to 0$, are approached exponentially fast as $\xi\to\pm \infty$, we have that both expressions in \eqref{Frver} decay exponentially fast to zero. It follows that the two Fr\'echet derivatives are in $\LT$.}}

So now, the solitary waves can now be considered as
constrained critical points of $\mathcal{G}$ subject to fixed $\mathcal{F}$.
Since $\mathcal{F}$ is conserved, it follows that the evolution of
\eqref{1} does not occur on all of $X_k$  but rather on the co-dimension one submanifold
$$\mathcal{M}=\{m\in X_k: \mathcal{F}(m)=\mathcal{F}(\mu)\}.$$
We define
\eq{\mathcal{T}\triangleq\left\{\widetilde{m}\in H^1(\mathbb{R}):  \left<\frac{\delta \mathcal{F}}{\delta m}(\mu),\widetilde{m}\right>=0\right\}.}{Taudef}
The set $\mathcal{T}$ is the tangent space in $H^1(\mathbb{R})$ to the submanifold $\mathcal{M}$ at the point $\mu$.
It also  represents the set of perturbations that satisfy the first-order linear condition for the nonlinear constraint $\mathcal{F}(\mu+\widetilde{m})=\mathcal{F}(\mu)$ to be satisfied.

Now we add one constraint aimed at ``removing'' the kernel and define the constraint space as
%$$
%\mathcal{S}\triangleq \left\{\widetilde{m}\in H^1(\mathbb{R}): \langle \widetilde{m},\mu_\xi\rangle=0,\mathcal{F}(\mu+\widetilde{m})=\mathcal{F}(\mu)\right\},
%$$
%with linear approximation
\eq{
\mathcal{S}_L\triangleq \left\{\widetilde{m}\in H^1(\mathbb{R}): \left\langle \widetilde{m},\mu_\xi\right\rangle{{=}}\left\langle\frac{\delta \mathcal{F}}{\delta m}(\mu),\widetilde{m}\right\rangle=0\right\}.
}{SLdef}
Under the spectral properties of $\mathcal{L}$, as a consequence of \cite{VK73}, the operator $\mathcal{L}$ is strictly
positive definite in $\mathcal{S}_L$ if and only if
\begin{equation}\label{4.3}
 \left\langle \mathcal{L}^{-1}\frac{\delta \mathcal{F}}{\delta m}(\mu),\frac{\delta \mathcal{F}}{\delta m}(\mu)\right\rangle<0.
\end{equation}
This thus implies that the solitary wave with profile $\mu$ is a local constrained
 minimizer of $\Lambda$ under the constraints defining $\mathcal{S}_L$ in \eqref{SLdef}.
%\begin{equation}\label{4.4}
% \left\langle\frac{\delta \mathcal{F}}{\delta m}(\mu),\widetilde{m}\right\rangle=0
%\end{equation}

\begin{remark} \label{re4.1}
The condition \eqref{4.3} guaranteeing positive definiteness of $\mathcal{L}$ on $\mathcal{S}_L$ is referred
to as the so-called Vakhitov-Kolokolov condition. Later we will derive a analytical
representation for the inner-product in \eqref{4.3}. As a matter of fact, we will prove that the condition  \eqref{4.3} always holds for all smooth solitary waves $\mu$ constructed in Lemma \ref{le3.1}.
\end{remark}

\begin{remark} \label{re4.2}
By \eqref{3.11} and \eqref{Frver}, the constraint used to define the set $\mathcal{T}$ in \eqref{Taudef} can be equivalently written as
\eqnn{
 \langle \phi-k,\widetilde{m}\rangle=0,
}
where $\phi-k$ decays to zero at infinity exponentially
fast.
\end{remark}

\subsection {Analysis of the Vakhitov-Kolokolov condition}
\label{S:VK}
In this subsection, we seek to derive an analytic representation for the condition \eqref{4.3}.

{{We first briefly show that the solitary wave depends smoothly on the parameters $c$, $a$, and $E$, thus justifying the derivatives we are taking in this section.  From \eqref{3.3}, which is satisfied by the solitary wave, and using the reasoning that led to \eqref{3.9} {{where the positive branch of the square root was identified as the correct one for this solution,}} we obtain
\[
\psi = \phi_\xi = \phi^2 - c + \sqrt{c^2 + E - 4a\phi} = \phi^2 - c + \sqrt{(c - k^2)(c + 3k^2 - 4k\phi)}.
\]
Here, the second form follows from the computation leading to \eqref{3.9}, or equivalently by applying \eqref{3.4}.

Since Lemma~\ref{le3.1} ensures that \(k < \phi < \sqrt{c}\) and the maximum of \(\phi\) occurs when \(\psi = 0\), the expression under the square root is strictly positive. Consequently, the right-hand side of the differential equation defining \(\phi_\xi\) is a smooth function of \(\phi\) and the parameters \(c\), \(a\), and \(E\). This implies, by standard results on differential equations,  that \(\phi\) itself depends smoothly on these three parameters.}}

Recall that for the smooth solitary wave $u=\phi$, $\mu=\phi-\phi_{\xi\xi}$ is a critical point of $\Lambda$, that is,
$$
\frac{\delta \Lambda}{\delta m}(\mu)= \frac{\delta F_3}{\delta m}(\mu)+\omega_1\frac{\delta F_1}{\delta m}(\mu)+\omega_2\frac{\delta F_2}{\delta m}(\mu)=0.
$$
Differentiating this relation  with respect to the parameters $a, E$, and $c$, yields
\eq{
\mathcal{L}\mu_a=-\frac{\partial \omega_1}{\partial a}\frac{\delta F_1}{\delta m}(\mu)-\frac{\partial \omega_2}{\partial a}\frac{\delta F_2}{\delta m}(\mu),\\
\mathcal{L}\mu_E=-\frac{\partial \omega_1}{\partial E}\frac{\delta F_1}{\delta m}(\mu)-\frac{\partial \omega_2}{\partial E}\frac{\delta F_2}{\delta m}(\mu),\\
\mathcal{L}\mu_c=-\frac{\partial \omega_1}{\partial c}\frac{\delta F_1}{\delta m}(\mu)-\frac{\partial \omega_2}{\partial c}\frac{\delta F_2}{\delta m}(\mu).
}{4.6}

The following result is due to a scaling symmetry of the profile equation.

\begin{lemma} \label{le4.1}
 {{Let $u=\phi$ be the smooth traveling wave solution of the mCH equation \eqref{1}.}} Then $\mu=\phi-\phi_{\xi\xi}$ satisfies
\begin{equation}\label{4.7}
3a\mu_a+4E\mu_E+2c\mu_c=\mu.
\end{equation}
\end{lemma}

\begin{proof}
Note that the solution $\phi(x;a,E,c)$ of the the profile equations \eqref{3.1} and \eqref{3.3} enjoy the scaling symmetry
$$
\phi(\xi;a,E,c)=c^{\frac{1}{2}}\psi(\xi;\mathbb{A},\mathbb{E},c), \quad a=c^{\frac{3}{2}} \mathbb{A},\ \  E=c^2  \mathbb{E}
$$
where $\psi$ and the constants $\mathbb{A}$ and $\mathbb{E}$ are independent of $c$. Differentiating with respect to $c$ gives
$$
\frac{\partial \phi}{\partial a}\frac{\partial a}{\partial c}+\frac{\partial \phi}{\partial E}\frac{\partial E}{\partial c}+\frac{\partial \phi}{\partial c}=\frac{1}{2\sqrt{c}}\psi=\frac{1}{2c}\phi.
$$
Since
$$
\frac{\partial a}{\partial c}=\frac{3}{2}c^{\frac{1}{2}}\mathbb{A}=\frac{3a}{2c}, \quad \frac{\partial E}{\partial c}=2c\mathbb{E}=\frac{2E}{c}
$$
it follows that
$$
3a\frac{\partial \phi}{\partial a}+4E\frac{\partial \phi}{\partial E}+2c\frac{\partial \phi}{\partial c}=\phi.
$$
Noting that $\mu=\phi-\phi_{\xi\xi}$, it follows that the same expression holds for $\mu$, as claimed.
\end{proof}

By using the identities \eqref{4.6} and \eqref{4.7}, we have
\begin{eqnarray*}
\mathcal{L}\mu &=& 3a\mathcal{L}\mu_a+4E\mathcal{L}\mu_E+2c\mathcal{L}\mu_c\\
&=& -\left(3a\frac{\partial \omega_1}{\partial a}+4E\frac{\partial \omega_1}{\partial E}+2c\frac{\partial \omega_1}{\partial c}\right)\frac{\delta F_1}{\delta m}(\mu)-\left(3a\frac{\partial \omega_2}{\partial a}+4E\frac{\partial \omega_2}{\partial E}+2c\frac{\partial \omega_2}{\partial c}\right)\frac{\delta F_2}{\delta m}(\mu).
\end{eqnarray*}
Furthermore, differentiating the equation $\frac{\delta \Lambda}{\delta m}(\mu)=0$ with respect to $k$ yields
$$
\mathcal{L}\mu_k=-\frac{\partial \omega_1}{\partial k}\frac{\delta F_1}{\delta m}(\mu)-\frac{\partial \omega_2}{\partial k}\frac{\delta F_2}{\delta m}(\mu).
$$
Since $\mu(x)\rightarrow k$ as $|x|\rightarrow \infty$, then $\mu_k(x)\rightarrow 1$
as $|x|\rightarrow \infty$, so that $\mu_k(x)$ does not decay to $0$ at infinity.
However, $k\mu_k(x)-\mu(x)$ does decay to $0$ at infinity, so we have
 \begin{eqnarray*}\label{11}
 \mathcal{L}(k\mu_k-\mu) &=&  \left(-k\frac{\partial \omega_1}{\partial k}+3a\frac{\partial \omega_1}{\partial a}+4E\frac{\partial \omega_1}{\partial E}+2c\frac{\partial \omega_1}{\partial c}\right)\frac{\delta F_1}{\delta m}(\mu) \nonumber\\
& & +\left(-k\frac{\partial \omega_2}{\partial k}+3a\frac{\partial \omega_2}{\partial a}+4E\frac{\partial \omega_2}{\partial E}+2c\frac{\partial \omega_2}{\partial c}\right)\frac{\delta F_2}{\delta m}(\mu).
\end{eqnarray*}
By using \eqref{3.2}, \eqref{3.13} and \eqref{3.18}, we get
 \begin{eqnarray*}
 \mathcal{L}(k\mu_k-\mu) &=&  \frac{4c}{k^2(c-k^2)^2}\frac{\delta F_1}{\delta m}(\mu)+\frac{4c}{(c-k^2)^2}\frac{\delta F_2}{\delta m}(\mu)\\
&=&   \frac{4c}{(c-k^2)^2}\left[\frac{1}{k^2}\frac{\delta F_1}{\delta m}(\mu)+\frac{\delta F_2}{\delta m}(\mu)\right]\\
&=& \frac{4c}{(c-k^2)^2}\frac{\delta \mathcal{F}}{\delta m}(\mu).
\end{eqnarray*}

{{
{From the preceding equation, it follows that }
${\delta \mathcal{F}}/{\delta m}$ is in the range of  $\mathcal{L}$.
Moreover, since Lemma~\ref{le3.1} implies that $\mu$ is an even function of $\xi$,
the quantity $k\mu_k - \mu$ is even and thus lies in the orthogonal complement of the one-dimensional kernel of $\mathcal{L}$,
which is generated by the odd function $\mu_\xi$.
Given the spectral properties of $\mathcal{L}$ described in Proposition~\ref{pro3.1},
its inverse exists if the domain of $\mathcal{L}$ is restricted on this subspace.
Therefore, we conclude that}}
$$
\mathcal{L}^{-1}\frac{\delta \mathcal{F}}{\delta m}(\mu)=\frac{(c-k^2)^2}{4c}(k\mu_k-\mu)
$$
and then the condition \eqref{4.3} can be rewritten as
\begin{equation}\label{4.8}
 \left\langle \mathcal{L}^{-1}\frac{\delta \mathcal{F}}{\delta m}(\mu),\frac{\delta \mathcal{F}}{\delta m}(\mu)\right\rangle=\frac{(c-k^2)^2}{4c}\int_{\mathbb{R}}\frac{\delta \mathcal{F}}{\delta m}(\mu)(k\mu_k-\mu)d\xi.
\end{equation}

With this representation, we can establish the following stability criterion.

\begin{lemma} \label{le4.2}
The Vakhitov-Kolokolov condition \eqref{4.3} holds if and only if
\begin{equation}\label{4.9}
\frac{d}{dk}Q(\phi,k)<0,\quad Q(\phi,k)=\int_{\mathbb{R}}\left(\sqrt{\frac{c-k^2}{c+3k^2-4k\phi}}+\sqrt{\frac{c+3k^2-4k\phi}{c-k^2}}-2\right)d\xi,
\end{equation}
where $\phi=\phi(\xi)$ is the solitary wave solution obtained in Lemma \ref{le3.1} which relates $\mu$ through \eqref{3.11}.
\end{lemma}

\begin{proof}
In view of \eqref{4.1}, we have
$$
\frac{d}{dk}\mathcal{F}(\mu)=\int_{\mathbb{R}}\frac{\delta \mathcal{F} }{\delta m}(\mu)\mu_kd\xi-2k^{-3}F_1(\mu)
$$
and this leads to
$$
\int_{\mathbb{R}}\frac{\delta \mathcal{F}}{\delta m}(\mu)(k\mu_k-\mu)d\xi=k\frac{d}{dk}\mathcal{F}(\mu)+2k^{-2}F_1(\mu)-\int_{\mathbb{R}}\frac{\delta \mathcal{F}}{\delta m}(\mu)\mu d\xi.
$$
On the other hand, by using \eqref{5} and \eqref{6}, we get
\begin{eqnarray*}
\int_{\mathbb{R}}\frac{\delta \mathcal{F}}{\delta m}(\mu)\mu d\xi &=&\int_{\mathbb{R}}\left[\frac{\delta F_2}{\delta m}(\mu)+k^{-2}\frac{\delta F_1}{\delta m}(\mu)\right]\mu d\xi\\
&=& \int_{\mathbb{R}}(-\mu^{-2}+k^{-2})\mu d\xi\\
&=& \int_{\mathbb{R}}\left[-(\mu^{-1}-k^{-1})+k^{-2}(\mu-k)\right]d\xi\\
&=& -F_2(\mu)+k^{-2}F_1(\mu).
\end{eqnarray*}
Thus, it follows that
\begin{eqnarray}\label{4.10}
\int_{\mathbb{R}}\frac{\delta \mathcal{F}}{\delta m}(\mu)(k\mu_k-\mu) d\xi &=& k\frac{d}{dk}\mathcal{F}(\mu)+F_2(\mu)+k^{-2}F_1(\mu)\nonumber\\
&=&  k\frac{d}{dk}\mathcal{F}(\mu)+\mathcal{F}(\mu)\nonumber\\
&=&  \frac{d}{dk}[k\mathcal{F}(\mu)],
\end{eqnarray}
where we used the definition of $\mathcal{F}$ given in \eqref{4.1}.
Making use of \eqref{5}, \eqref{6} and \eqref{3.11}, we have
\begin{eqnarray}\label{4.11}
\frac{d}{dk}[k\mathcal{F}(\mu)] &=& \frac{d}{dk} [kF_2(\mu)+k^{-1}F_1(\mu)]\nonumber\\
&=&  \frac{d}{dk} \int_{\mathbb{R}}\left(\frac{k}{\mu}+\frac{\mu}{k}-2\right)d\xi\nonumber\\
&=&  \frac{d}{dk} \int_{\mathbb{R}}\left(\sqrt{\frac{c-k^2}{c+3k^2-4k\phi}}+\sqrt{\frac{c+3k^2-4k\phi}{c-k^2}}-2\right)d\xi\nonumber\\
&= & \frac{d}{dk}Q(\phi,k).
\end{eqnarray}
Thus, the statement of Lemma 4.2 follows from \eqref{4.8}, \eqref{4.10} and \eqref{4.11}.
\end{proof}

\subsection {Verification of the stability criterion \eqref{4.9}}
In this subsection, we give a proof that the stability criterion \eqref{4.9} is satisfied. To this end, we introduce a change of variables that simplifies the first-order system \eqref{3.2} and then, as a consequence, the stability criterion.

Firstly, we rescale the smooth solitary wave solutions $\phi$ as
\begin{equation}\label{4.12}
 \phi=k+\beta \varphi, \quad \beta:=\frac{c-k^2}{4k},
\end{equation}
where $\beta\in(\frac{\sqrt{3c}}{6},\frac{2\sqrt{c}}{3})$ due to $k\in(\frac{\sqrt{c}}{3},\frac{\sqrt{3c}}{3})$. By Lemma \ref{le3.1}, we have that
\eq{\varphi\in (0,\varphi_0]\subset(0,1),\;\;\varphi_0:=\sup_{\xi\in\mathbb{R}}{(\varphi)}=\frac{4k\left(\sqrt{2(c-k^2)}-2k\right)}{c-k^2},}{p0def}
 and
it is easily seen that $\varphi_0\rightarrow 1$ as $k\rightarrow \frac{\sqrt{c}}{3}$. Note that the fact that $\varphi_0<1$ for $k\in(\frac{\sqrt{c}}{3},\frac{\sqrt{3c}}{3})$ follows from the fact that
$$
\varphi_0<1\iff \left(\sqrt{c-k^2}-2\sqrt{2}k\right)^2>0.
$$
Substituting \eqref{4.12} into the first-order system \eqref{3.2}, and using the expression giving $a$ in \eqref{3.4}, yields
\begin{equation}\left\{\begin{array}{l}\label{4.13}
\displaystyle{\varphi_\xi=\widetilde{\psi},} \\ \\
 \displaystyle{\widetilde{\psi}_\xi=\varphi+\frac{k\beta(\varphi^2-\widetilde{\psi}^2)+2k^2 \varphi}{k^2-c+\beta^2(\varphi^2-\widetilde{\psi}^2)+2k\beta \varphi},}
\end{array}\right.\end{equation}\\
which has the following first integral
\begin{equation}\label{4.14}
 \widetilde{H}(\varphi,\widetilde{\psi}):=\left[\beta^2(\varphi^2-\widetilde{\psi}^2)+2k\beta \varphi+k^2-c\right]^2+(c-k^2)^2\varphi-(c-2k^2)^2.
\end{equation}
The integral above can be obtained (up to the addition of a constant) from the first integral \eqref{3.3} using the change of variables \eqref{4.12}, or it can be obtained directly from System \eqref{4.13} by integration.
System \eqref{4.13} possesses two equilibria: the saddle point $(0, 0)$ and the center $(\frac{\sqrt{4c-3k^2}-3k}{2\beta}, 0)$, where $0<\frac{\sqrt{4c-3k^2}-3k}{2\beta}<1$.
Note that along the homoclinic orbit of \eqref{4.13},  the Hamiltonian energy is given by $ \widetilde{H}(\varphi,\widetilde{\psi})= \widetilde{H}(0,0)\triangleq \widetilde{H}_0=k^2(2c-3k^2)$.

Next, we write the equation for the level set of the Hamiltonian function \eqref{4.14}
with the energy $\widetilde{H}_0=k^2(2c-3k^2)$ as
$$
\beta^2(\varphi^2-\widetilde{\psi}^2)+\frac{c-k^2}{2}\varphi+k^2-c+(c-k^2)\sqrt{1-\varphi}=0,
$$
which can, with the help of \eqref{4.12}, be rewritten  as
\begin{equation}\label{4.15}
 \frac{\varphi^2-\widetilde{\psi}^2}{2-\varphi-2\sqrt{1-\varphi}}=\frac{8k^2}{c-k^2}.
\end{equation}

\begin{lemma} \label{le4.4}
The criterion \eqref{4.9} always holds for $\varphi\in(0,1)$.
\end{lemma}

\begin{proof}
It follows from \eqref{4.15} that
\eqnn{
\varphi_{\xi}^2=\varphi^2-h(1-\sqrt{1-\varphi})^2,
}
where $h=\frac{8k^2}{c-k^2}$ and $\varphi\in(0,1)$. Note that for $\xi\in(-\infty,0)$, $\varphi_\xi>0$, so we have
 \begin{equation}\label{4.28}
 \varphi_{\xi}=\sqrt{\varphi^2-h\left(1-\sqrt{1-\varphi}\right)^2},\quad \xi\in(-\infty,0).
\end{equation}

 Inserting \eqref{4.12} and \eqref{4.28} into the expression for $Q$ in \eqref{4.9}, and taking advantage of the evenness of $\varphi$, we then have
\begin{eqnarray}\label{4.29}
Q(\phi,k) &=&  \int_{\mathbb{R}}\left(\sqrt{\frac{c-k^2}{c+3k^2-4k\phi}}+\sqrt{\frac{c+3k^2-4k\phi}{c-k^2}}-2\right)d\xi\nonumber\\
&=&   \int_{\mathbb{R}}\left( \frac{1}{\sqrt{1-\varphi}}+\sqrt{1-\varphi}-2\right)d\xi\nonumber\\
&=&   2\int_{-\infty}^{0}\frac{(1-\sqrt{1-\varphi})^2}{\sqrt{1-\varphi}}d\xi\nonumber\\
&=&   2\int_{0}^{\varphi_0}\frac{(1-\sqrt{1-\varphi})^2}{\sqrt{1-\varphi}}\cdot \frac{d\varphi}{\sqrt{\varphi^2-h\left(1-\sqrt{1-\varphi}\right)^2}}.
\end{eqnarray}
Making the change of variable
$$ t\triangleq \sqrt{1-\varphi}, \quad \varphi=1-t^2,\quad d\varphi=-2tdt,$$
with $h=\frac{8k^2}{c-k^2}$ and  $\varphi_0$ given in \eqref{p0def}, the expression of $Q$ in \eqref{4.29} becomes
 \begin{eqnarray*}
Q(\phi,k) &=&  2\int_{1}^{\sqrt{1-\varphi_0}}\left[\frac{(1-t)^2}{t}\cdot \frac{-2t}{\sqrt{(1-t^2)^2-h(1-t)^2}}\right]dt\nonumber\\
&=&   4\int_{\sqrt{1-\varphi_0}}^{1}\left(\frac{1-t}{\sqrt{(1+t)^2-h}}\right)dt\nonumber\\
&=&   \left(8\ln(t+1+\sqrt{(t+1)^2-h})-4\sqrt{(t+1)^2-h}\right)\Bigg|_{t=\sqrt{1-\varphi_0}}^{t=1}\nonumber\\
&=&   8\ln\left(\frac{2+\sqrt{4-h}}{1+\sqrt{1-\varphi_0}}\right)-4\sqrt{4-h}\nonumber\\
&=&  8\ln\left(\frac{\sqrt{c-k^2}+\sqrt{c-3k^2}}{\sqrt{2}k}\right)-8\sqrt{\frac{c-3k^2}{c-k^2}}.
\end{eqnarray*}
A straightforward calculation yields that
$$
\frac{d}{dk}Q(\phi,k)=\frac{-8c}{k(c-k^2)}\sqrt{\frac{c-3k^2}{c-k^2}}<0.
$$
This completes the proof of the stability criterion \eqref{4.9}.
\end{proof}

\subsection {Proof of Theorem 1.1}

It follows from lemmas \ref{le4.2} and \ref{le4.4} that the Vakhitov-Kolokolov condition \eqref{4.3} holds.
As mentioned before, because of the spectral properties of $\mathcal{L}$ stated in Proposition \ref{pro3.1},
it follows from \cite{VK73} that $\mathcal{L}$ is positive definite in the subspace $\mathcal{S}_L$ of $\HT$
defined in \eqref{SLdef}. That is, for all $\widetilde{m}\in {\mathcal{S}}_L\subset H^1(\mathbb{R})$,
we have \eq{ \left\langle\mathcal{L}\widetilde{m},\widetilde{m}\right\rangle\geq C\|\widetilde{m}\|^2_{H^1(\mathbb{R})}
\text{ for some }C>0. }{PD} Our main result in Theorem \ref{th1.1} is derived from the orbital stability theory presented
in \cite[Section 3]{GSS87}. This result is a consequence of the variational characterization of the traveling wave given
in Lemma \ref{le3.2}, the expression for $\Lambda$ given in \eqref{OmFr} in terms of two Fr\'echet differentiable functionals,
the positive-definiteness of $\mathcal{L}$ on $\mathcal{S}_L$ as stated in \eqref{PD}, and the local well-posedness theory for
the mCH equation (\ref{1}) in $X_k$ as shown in Proposition {\ref{pro.1}.} One of the assumptions in \cite{GSS87} is that
the $\mathcal{J}$ defined in \eqref{HM} is onto. However, this condition is not used in the proof of stability
in \cite[Section 3]{GSS87}. The condition on the operator $\mathcal{J}$ is only employed in \cite{GSS87} to prove instability.
The reader is also encouraged to see the proof of \cite[Theorem 4.5]{EJL24}, in which the orbital stability is proven,
without the use of the Hamiltonian operator, in the context of the Novikov equation under the same conditions listed above
for our case.

 \qed

\noindent{\bf Acknowledgements}

This research was partially by the NSFC (No.12571172) and by the
Scientific Research Fund of Hunan Provincial Education Department
(No.21A0414), {{and Doctoral Research Fund of Hubei University of Automotive Technology (No.BK202434)}}. The research of S. Lafortune was supported by a Collaboration Grants for Mathematicians from the Simons Foundation (award \# 420847). The research of Z. Liu was supported by the NSFC (No.12571188) and the Fundamental Research Funds for the Central
Universities, China University of Geosciences (Nos.CUGST2), and Guangdong Basic and Applied Basic Research
Foundation (Nos.2023A1515011679; 2024A1515012704). One of the authors (S.L.) is grateful for helpful email discussions with Teng Long from the Huazhong University of Science and Technology.

\appendix
 \section {Conservation of the functional {\blue {$F_2$ and $F_3$}}}
 \label{AA}
{{In this Appendix, we show explicitly that the quantity {\blue {$F_2$ and $F_3$}} given in \eqref{6} and \eqref{7} is time-independent for the mCH equation \eqref{1}. We first do this by a direct computation and then we show that they both are Casimirs for the Hamiltonian operator $\mathcal{J}$ in \eqref{HM}.

% \subsection {Proof that the functionals $F_2$ and $F_3$ are conserved}
% \label{2.2s}
 First, we do this by a direct computation.
It follows from \eqref{1} that
\begin{equation}\label{2.2}
m_t=-[(u^2-u_x^2)m_x+2m^2u_x], \ \  m_{tx}=-[(u^2-u_x^2)m_{xx}+6u_xmm_x+2m^2(u-m)],
\end{equation}

A straightforward calculation shows that
\begin{eqnarray*}
\frac{d}{dt}F_2(m(t))&=& \frac{d}{dt}\int_{\mathbb{R}}\left(\frac{1}{m}-\frac{1}{k}\right)dx \\
&=& \int_{\mathbb{R}} \left[2u_x+\frac{(u^2-u_x^2)m_x}{m^2}\right]dx  \\
&=& -\int_{\mathbb{R}} (u^2-u_x^2)\left(\frac{1}{m}\right)_xdx\\
&=& \int_{\mathbb{R}} \frac{(u^2-u_x^2)_x}{m}dx\\
&=& \int_{\mathbb{R}}2u_xdx=0.
\end{eqnarray*}

Similarly, we have
\begin{eqnarray*}
\frac{d}{dt}F_3(m(t))&=& \frac{d}{dt}\int_{\mathbb{R}}\frac{m_x^2}{m^5}dx+\frac{d}{dt}\int_{\mathbb{R}}\frac{1}{4m^3}dx \\
&=& I_1+I_2.
\end{eqnarray*}
In view of \eqref{2.2}, and integrating by parts several times leads to
\begin{eqnarray*}
I_1&=&  \frac{d}{dt}\int_{\mathbb{R}}\frac{m_x^2}{m^5}dx\\
&=& \int_{\mathbb{R}} \left(\frac{2m_xm_{xt}}{m^5}-\frac{5m_x^2m_t}{m^6}\right)dx\\
&=& \int_{\mathbb{R}} \left(\frac{-2m_x[(u^2-u_x^2)m_{xx}+6mm_xu_x+2m^2(u-m)]}{m^5}+\frac{5m_x^2[(u^2-u_x^2)m_x+2m^2u_x]}{m^6}\right)dx\\
&=& \int_{\mathbb{R}} \left(\left[\frac{-m_x^2(u^2-u_x^2)}{m^5}\right]_x-\frac{4m_x(u-m)}{m^3}\right)dx\\
&=&  -\int_{\mathbb{R}} \left.\frac{4m_x(u-m)}{m^3}\right.dx,
\end{eqnarray*}
and
\begin{eqnarray*}
I_2&=&  \frac{d}{dt}\int_{\mathbb{R}}\frac{1}{4m^3}dx\\
&=& \int_{\mathbb{R}} \frac{3}{4m^4}[(u^2-u_x^2)m]_xdx\\
&=& \int_{\mathbb{R}}\left( \frac{3}{4}\left[\frac{u^2-u_x^2}{m^3}\right]_x+\frac{3(u^2-u_x^2)m_x}{m^4}\right)dx\\
&=& -\int_{\mathbb{R}}  (u^2-u_x^2)\left(\frac{1}{m^3}\right)_xdx\\
&=&  \int_{\mathbb{R}} \left(\frac{(u^2-u_x^2)_x}{m^3}-\left[\frac{u^2-u_x^2}{m^3}\right]_x\right)dx\\
&=&  \int_{\mathbb{R}} \left(\frac{2u_x}{m^2}\right)dx.
\end{eqnarray*}

Thus, we obtain
\begin{eqnarray*}
I_1+I_2&=& \int_{\mathbb{R}}\left( -\frac{4m_x(u-m)}{m^3}+\frac{2u_x}{m^2}\right)dx\\
&=& \int_{\mathbb{R}} \left( \frac{-4um_x}{m^3}-4\left(\frac{1}{m}\right)_x+\frac{2u_x}{m^2}\right)dx\\
&=& \int_{\mathbb{R}} \left( \frac{-4um_x}{m^3}+\frac{2u_x}{m^2}\right)dx\\
&=& \int_{\mathbb{R}}   \left(\frac{2u}{m^2}\right)_xdx=0.
\end{eqnarray*}
Consequently, it yields that $\frac{d}{dt}F_3(m(t))=0$.

As a second method for proving that $F_2$ and $F_3$ are conserved, we establish that they are Casimirs of the Hamiltonian operator $\mathcal{J}$ in~\eqref{HM}. The computation parallels the approach used in \cite[Appendix~A]{LP22} for identifying conserved quantities in the $b$-family.

Assume that $\mathcal{J}f = 0$, where $\mathcal{J}$ is given in~\eqref{HM}. Then
\eq{
\mathcal{J} f = \partial_x m \,\partial_x^{-1} m \, (\partial_x^2 - 1)^{-1} \, \partial_x m \,\partial_x^{-1} m \, \partial_x f = 0.
}{CasCond}

{{We use the representation from \cite[(2.4) and the two following equations]{Olver2016}, namely
\eq{
\mathcal{J} = \mathcal{K} \mathcal{B}^{-1} \mathcal{K}, \qquad
\mathcal{K} \equiv \partial_x m \,\partial_x^{-1} m \,\partial_x,\;\;
\mathcal{B} \equiv  \partial_x^3-\partial_x.
}{Jform}

To determine $f$, we use the fact that $F_2$ was previously identified as a Casimir of $\mathcal{K}$ in \cite[Equation~(2.6) and the discussion after (3.5)]{Olver2016} and \cite[discussion after (26)]{OR96}, i.e.
$$
\mathcal{K}\left(\frac{\delta F_2}{\delta m}\right)=0.
$$

From the representation of $\mathcal{J}$ in~\eqref{Jform}, it follows that $F_2$ is therefore a Casimir of $\mathcal{J}$. To construct another Casimir, we set
\eq{
(\mathcal{B}^{-1}\mathcal{K})f
= (\partial_x^2 - 1)^{-1} m\,\partial_x^{-1} m\,\partial_x f
= -\frac{\delta F_2}{\delta m}
= \frac{1}{m^2}.
}{CasCond2}

From~\eqref{CasCond2}, one computes directly that
\eqnn{
\partial_x f
= m^{-1}\,\partial_x m^{-1}(\partial_x^2 - 1)\!\left(\frac{1}{m^2}\right)=-\left(\left(\frac{2m_x}{m^5}\right)_x+\frac{5m_x^3}{m^6}+\frac{3m}{4m^4}\right)_x
= \left(\frac{\delta F_3}{\delta m}\right)_x.
}

Consequently,
\eqnn{
f = \frac{\delta F_3}{\delta m}
}
satisfies~\eqref{CasCond}. Thus, $F_2$ and $F_3$ are Casimirs of the operator $\mathcal{J}$ and therefore are conserved quantities.

Note that we avoided the issue of choosing the correct constant of integration when applying $\partial_x^{-1}$ by relying on \cite{Olver2016, OR96}, where it is established that $F_2$ is a Casimir of $\mathcal{K}$.}} This choice can alternatively be enforced by selecting appropriate integration limits when defining the operators, together with corresponding boundary conditions; see, for example, \cite[Equation~(2.2)]{C1}, which discusses Casimirs for nonlocal Hamiltonian operators, and also \cite[Section~2]{C2}. Additional discussion on nonlocal operators and symmetries appears in \cite{C3, C4}. {{As in other peakon-type equations, the inverse operator { {$(1 - \partial_x^2)^{-1}f$ }}is defined by convolution: $(e^{-|x|}/2) * f$.}}

\end{document}